\documentclass[12pt]{amsart}
\usepackage{bm}
\usepackage[dvipsnames,svgnames,x11names]{xcolor}
\usepackage[hyphens]{url}
\usepackage[backref=page,colorlinks=true,linkcolor=Maroon,citecolor=blue,urlcolor=blue,hypertexnames=false,linktocpage]{hyperref}
\usepackage{bookmark}
\usepackage{amsmath,thmtools,mathtools}
\mathtoolsset{showonlyrefs=true}
\usepackage{blindtext}
\usepackage{tikz}
\usepackage{fancyhdr}
\usepackage{esint}
\usepackage{enumerate}
\usepackage{pictexwd,dcpic}
\usepackage{graphicx}
\usepackage{caption}
\allowdisplaybreaks

\newcounter{mgncount}

\newtheorem*{claim}{Claim}

\newtheorem*{qu-}{Question}
\declaretheorem[name=Theorem,numberwithin=section]{thm}

\declaretheorem[name=Lemma,sibling=thm]{lemma}

\numberwithin{equation}{section}

\newcommand{\bbN}{\mathbb{N}}

\newcommand{\bbR}{\mathbb{R}}

\newcommand{\bbS}{\mathbb{S}}

\newcommand{\al}{\alpha}

\newcommand{\cC}{\mathcal{C}}

\newcommand{\cK}{\mathcal{K}}

\newcommand{\cM}{\mathcal{M}}

\newcommand{\cT}{\mathcal{T}}

\newcommand{\eq}[1]{\begin{equation}\begin{alignedat}{2} #1 \end{alignedat}\end{equation}}

\newcommand{\q}{\quad}

\begin{document}

\title[Capillary $L_p$ Minkowski Flows]{Capillary $L_p$ Minkowski Flows}
	\author[J. Hu, Y. Hu, M. N. Ivaki]{Jinrong Hu, Yingxiang Hu, Mohammad N. Ivaki}

\begin{abstract}
We study the long-time existence and asymptotic behavior of a class of anisotropic capillary Gauss curvature flows. As an application, we provide a flow approach to the existence of smooth solutions to the capillary even $L_p$ Minkowski problem in the Euclidean half-space for all $p \in (-n-1, \infty)$ and capillary $L_p$ Minkowski problem for $p > n+1$.
\end{abstract}

\keywords{Capillary $L_p$ Minkowski problem, capillary Gauss curvature flows}
\subjclass[2020]{53C21, 35J66, 35K55, 52A20}
\thanks{}
\maketitle

\section{Introduction}
\label{Sec1}

The Minkowski problem, a fundamental problem in the Brunn-Minkowski theory of convex geometry, seeks to reconstruct a convex body from its surface area measure. It asks whether a given positive Borel measure on the unit sphere $\mathbb{S}^n$ arises as the surface area measure of a convex body in $\mathbb{R}^{n+1}$, and if so, whether this convex body is uniquely determined (up to translation). This problem was solved by Minkowski for polytopes (see \cite{M897, M903}) and subsequently by Aleksandrov \cite{A38, A39}, Nirenberg \cite{N53}, Pogorelov \cite{P78}, Cheng-Yau \cite{CY76}, and others.

The $L_p$ Minkowski problem, which concerns the prescription of the $L_p$ surface area measure, was first introduced by Lutwak \cite{L93} and solved by Lutwak and Oliker \cite{LO95} for the even regular case when $p > 1$. Building on Lutwak's foundational work, the $L_p$ Minkowski problem has seen significant progress; see, e.g., \cite{ACW01, Sta02, LYZ04, CW06, BLYZ13, HL13, LW13, Zh15a, Zhu15b, B19, GLW22} and the references therein.

With the development of the classical and $L_p$ Minkowski problems, analogous questions have naturally arisen in the geometric study of capillary hypersurfaces which are discussed in detail below.

Let $\{E_i\}_{i=1}^{n+1}$ denote the standard orthonormal basis of $\mathbb{R}^{n+1}$, and define the Euclidean half-space as
\eq{
\mathbb{R}^{n+1}_{+} = \{ y \in \mathbb{R}^{n+1} : y \cdot E_{n+1} > 0 \}.
}
Let $\Sigma \subset \overline{\mathbb{R}^{n+1}_{+}}$ be a smooth, properly embedded, compact hypersurface with boundary $\partial \Sigma \subset \partial \mathbb{R}^{n+1}_{+}$. We say that $\Sigma$ is a capillary hypersurface with constant contact angle $\theta \in (0, \pi)$ if it satisfies
\eq{
\nu \cdot e = \cos(\pi - \theta) \quad \text{along $\partial \Sigma$},
}
where $e := -E_{n+1}$ and $\nu$ denotes the outer unit normal of $\Sigma$ in $\overline{\mathbb{R}^{n+1}_{+}}$. We call $\widehat{\Sigma}$ a \emph{capillary convex body} if $\widehat{\Sigma}$ is a bounded, closed region in $\overline{\mathbb{R}^{n+1}_{+}}$ enclosed by a strictly convex capillary hypersurface $\Sigma$ and the support hyperplane $\partial \mathbb{R}^{n+1}_{+}$. 

Let $\cC_\theta = \cC_{\theta,1}$ be the capillary spherical cap of radius $1$ intersecting the support hyperplane $\partial \mathbb{R}^{n+1}_{+}$ at a constant angle $\theta \in (0,\pi)$, which is defined as \eqref{s2:def-C_theta}.

Let $\tilde{\nu}: \Sigma \to \mathcal{C}_\theta$ denote the capillary Gauss map of $\Sigma$, given by 
\eq{ 
\tilde{\nu} := \nu + \cos\theta\, e,
}
which is a diffeomorphism. The capillary area measure of $\Sigma$ is defined as
\eq{
S^{c}(\eta) = \int_{\tilde{\nu}^{-1}(\eta)} \big(1 + \cos \theta\, (\nu \cdot e)\big) \, d \mathcal{H}^{n}
}
for all Borel sets $\eta \subset \mathcal{C}_\theta$, where $\tilde{\nu}^{-1}: \cC_\theta \to \Sigma$ is the inverse capillary Gauss map, and $\mathcal{H}^{n}$ denotes the $n$-dimensional Hausdorff measure. The capillary area measure is a capillary counterpart (or Robin boundary analogue) of the surface area measure. By this diffeomorphism, $S^{c}(\eta)$ has the equivalent form
\eq{
S^{c}(\eta) = \int_{\eta} \frac{\sin^{2}\theta + \cos \theta\, (\xi \cdot e)}{\cK(\tilde{\nu}^{-1}(\xi))} \, d\xi,
}
where $\cK$ is the Gauss curvature of $\Sigma$ and $d\xi$ is the area element of $\mathcal{C}_{\theta}$. Put
\eq{
\ell(\xi) = \sin^{2}\theta + \cos \theta\, (\xi \cdot e).
}
Then $\ell$ is the capillary support function of $\cC_\theta$ (see \eqref{s2:capillary-spt-fucntion-C} below).

Let $f: \mathcal{C}_{\theta} \to (0,\infty)$ be a smooth function. The \emph{regular capillary Minkowski problem} asks for the existence of a solution $h: \mathcal{C}_\theta \to \mathbb{R}$ (which serves as the capillary support function of a strictly convex capillary hypersurface; see \autoref{Sec2} for more details) to the following Monge-Amp\`{e}re type equation with a Robin boundary condition:
\eq{
\left\{
\begin{aligned}
\det(\nabla^{2}h + h I) &= f, & \quad \text{in } \mathcal{C}_\theta, \\
\nabla_{\mu} h &= \cot \theta\, h, & \quad \text{on } \partial\mathcal{C}_\theta.
\end{aligned}
\right.
}
Here $\nabla h$ and $\nabla^{2} h$ denote the gradient and the Hessian of $h$ with respect to a local orthonormal frame on $\mathcal{C}_\theta$, and $\mu$ is the unit outward normal of $\partial\mathcal{C}_\theta \subset \mathcal{C}_\theta$.

The necessary and sufficient condition for solving the capillary Minkowski problem when $\theta \in (0,\frac{\pi}{2}]$ was established in \cite{MWW25} via the continuity method. The case $\theta > \frac{\pi}{2}$ remains open.

The capillary $L_p$-surface area measure $dS^{c}_{p}$ for a capillary convex body $\widehat{\Sigma}$ in $\overline{\mathbb{R}^{n+1}_{+}}$ is defined by
\eq{
dS^{c}_{p} = \frac{\ell\, h^{1-p}}{\cK} \, d\xi.
}

The \emph{regular capillary $L_p$ Minkowski problem} amounts to solving the following Monge-Amp\`{e}re type equation with a Robin boundary condition:
\eq{\label{Lp-MP-smooth}
\left\{
\begin{aligned}
\det(\nabla^{2}h + h I) &= fh^{p-1}, & \quad \text{in } \mathcal{C}_\theta, \\
\nabla_{\mu} h& = \cot \theta\, h, & \quad \text{on } \partial\mathcal{C}_\theta.
\end{aligned}
\right.
}

We say that a smooth function $f: \cC_\theta \to \mathbb{R}$ is \emph{even} if
\eq{
f(-\xi_1, \ldots, -\xi_n, \xi_{n+1}) = f(\xi_1, \ldots, \xi_n, \xi_{n+1}), \quad \forall\, \xi \in \mathcal{C}_\theta.
}
A strictly convex capillary hypersurface is called even if its capillary support function is even.

For the regular capillary $L_p$ Minkowski problem with $\theta \in (0, \frac{\pi}{2})$, Mei, Wang, and Weng solved the case $p = 1$ in \cite{MWW25}, the even case (i.e. when $f$ is even) for $1 < p < n+1$, and the case for $p \geq n+1$ in \cite{MWW25a}. Y. Hu and Ivaki \cite{HI25} solved the even case for $-n-1 < p < 1$ using an iterative scheme; see also \cite{I20}.

In comparison with the methods in \cite{HI25,MWW25,MWW25a}, we here study the capillary $L_p$ Minkowski problem for $p \in (-n-1,\infty)$  through anisotropic capillary Gauss curvature flows (see \eqref{GX}, \eqref{GX-p>n+1}). In particular, we solve the capillary even $L_p$ Minkowski problem for \emph{all} $p>-n-1$,  and the non-even case for $p>n+1$. It would be interesting to further solve the non-even case for $p=1$ and $p=n+1$ using geometric flows. 

Geometric flows have proven to be an efficient tool in the study of Minkowski-type problems. For example, the logarithmic Gauss curvature flow \cite{CW00} and its $L_p$ analogues, such as those in \cite{BIS19,LSW20,CL21,GLW22,BG25}, have been instrumental in analyzing both the classical and $L_p$ settings. The flows \eqref{GX}, \eqref{GX-p>n+1}, introduced shortly, may be viewed as capillary analogues of the $L_p$ Minkowski flows subject to Robin boundary conditions. Recently, an isotropic capillary Gauss curvature flow was considered by Mei, Wang, and Weng \cite{MWW25b}, which is a capillary counterpart of Firey's variational worn stone; see Firey \cite{F74}, Chou \cite{CH85}, Andrews \cite{And99}, and Guan and Ni \cite{GN17}. While the classification of solitons has been settled for the classical Gauss curvature flow, it remains open in the capillary case; see \cite{Gag84, And99, BCD17, Sar22, IM23, HI25a}. See also for additional work in the capillary framework \cite{HWYZ24, WWX24, KLS25, MWWX25}.

Let $Y_0: \mathcal{M} \to \overline{\mathbb{R}^{n+1}_{+}}$ be an embedding of an even, smooth, strictly convex capillary hypersurface $\cM_0 = Y_0(\mathcal{M})$. Consider a family of even, smooth, strictly convex capillary hypersurfaces $\cM_t := Y(\mathcal{M},t)$ satisfying
\eq{
\label{eq:un-normalized_flow}
\left\{
\begin{aligned}
\partial_t Y &= -\frac{f\circ \tilde{\nu}\, (Y \cdot \nu)^p}{1 + \cos\theta \, (\nu \cdot e)} \mathcal{K} \tilde{\nu}, & \quad \text{in $\mathcal{M} \times [0,\cT)$}, \\
\tilde{\nu} \cdot e &= 0, & \quad \text{on $ \partial \mathcal{M} \times [0,\cT)$}, \\
Y(\cdot,0) &= Y_0(\cdot), & \quad \text{in $\mathcal{M}$}.
\end{aligned}
\right.
}
The rescaled solution
\begin{equation}
X(\xi, \tau(t)) := \left( \frac{V(\widehat{\mathcal{M}}_0)}{V(\widehat{\mathcal{M}}_t)} \right)^{\frac{1}{n+1}} Y(\xi, t)
\end{equation}
with the new time parameter defined via
\begin{equation}
\frac{d\tau}{dt} = -\frac{1}{n+1} \frac{d}{dt} \log V(\widehat{\mathcal{M}}_t),\q \tau(0)=0,
\end{equation}
satisfies the following equations
\begin{equation}
\label{GX}
\scalebox{0.95}{
$\left\{
\begin{aligned}
\partial_{\tau}X
&= - \frac{(n+1) V(\widehat{\cM_0})}{\int_{\mathcal{C}_\theta} f h^{p}\, d\xi} \,
\frac{f \circ \tilde{\nu} \,(X \cdot \nu)^{p}}{1 + \cos \theta\, (\nu \cdot e)} \cK \tilde{\nu} + X, & \text{in } \mathcal{M} \times [0,\infty), \\
\tilde{\nu} \cdot e &= 0, & \text{on } \partial \mathcal{M} \times [0,\infty), \\
X(\cdot,0) &= X_0(\cdot), & \text{in } \mathcal{M}.
\end{aligned}
\right.$
}
\end{equation}
Here $\Sigma_{\tau} := X(\mathcal{M},\tau)$. Note that $\widehat{\Sigma_0}=\widehat{\cM_0}$ and $V(\widehat{\Sigma_{\tau}})=V(\widehat{\Sigma_0})$ for all $\tau>0$. We differentiate the two flows using distinct time parameters: $t$ is used for the shrinking flow, while $\tau$ is used for the volume-normalized flow. The asymptotic behavior of the rescaled solution is captured in the following theorem.

\begin{thm}\label{MTO}
Let $p > -n-1$ and $\theta \in (0,\frac{\pi}{2})$. Let $\Sigma_0$ be an even, smooth, and strictly convex capillary hypersurface in $\overline{\mathbb{R}^{n+1}_{+}}$. Let $f$ be an even, smooth, positive function on $\mathcal{C}_\theta$. Then there exists an even, smooth, and strictly convex solution $\Sigma_{\tau}$ to the flow \eqref{GX} for all $\tau > 0$. Moreover, a subsequence of $\{\Sigma_{\tau}\}$ converges in $C^{\infty}$ to an even, smooth, and strictly convex capillary hypersurface $\Sigma$ satisfying
\eq{\label{s1:static-eq}
\left\{
\begin{aligned}
\det(\nabla^{2}h + h I) &= \frac{(n+1) V(\widehat{\Sigma_{0}}) }{\int_{\cC_\theta} f h^{p} \, d\xi}f h^{p-1} , & \quad \text{in } \mathcal{C}_\theta, \\
\nabla_{\mu} h &= \cot \theta\, h, & \quad \text{on }\partial\mathcal{C}_\theta.
\end{aligned}
\right.
}
\end{thm}

Compared with the elliptic approach of \cite{MWW25,MWW25a} (see also \cite{LTU86,MQ19}), where the global $C^{2}$ estimate is reduced to the boundary double normal estimate and then derived via the construction of a suitable test function, we obtain the $C^{2}$ estimate here in a much simpler and more direct way.

The evenness restriction in \autoref{MTO} is only used to establish the $C^0$ estimate. In particular, if $p > n+1$, this assumption can be removed by employing the following anisotropic Gauss curvature flow
\eq{\label{GX-p>n+1}
\left\{
\begin{aligned}
\partial_{\tau} X
&= -\frac{f\circ \tilde{\nu} \, (X \cdot \nu)^{p}}{1 + \cos \theta\, (\nu \cdot e)} 
 \cK\, \tilde{\nu} + X, & \quad \text{in $ \cM \times [0,\infty)$}, \\
\tilde{\nu} \cdot e &= 0, & \quad \text{on } \partial \cM \times [0,\infty), \\
X(\cdot,0) &= X_0(\cdot), & \quad \text{in $\cM$}.
\end{aligned}
\right.
}
Compared to the flow \eqref{GX}, we dropped the factor $\frac{(n+1) V(\widehat{\Sigma_{0}})}{\int_{\mathcal{C}_\theta} f h^{p}\, d\xi}$ from the speed of the flow; see also \cite{BIS21a,BIS21b}.

\begin{thm}\label{MTO-noneven}
Let $p > n+1$ and $\theta \in (0,\frac{\pi}{2})$. Let $\Sigma_0$ be a smooth, strictly convex capillary hypersurface in $\overline{\mathbb{R}^{n+1}_{+}}$ with positive capillary support function. Let $f$ be a smooth, positive function on $\mathcal{C}_\theta$. Then there exists a smooth, strictly convex solution $\Sigma_{\tau}$ to the flow \eqref{GX-p>n+1} for all $\tau > 0$, and a subsequence of $\{\Sigma_{\tau}\}$ converges in $C^{\infty}$ to a smooth and strictly convex capillary hypersurface $\Sigma$ satisfying \eqref{Lp-MP-smooth}.
\end{thm}

The structure of this paper is as follows. In \autoref{Sec2}, we recall some basic facts about capillary hypersurfaces in the Euclidean  half-space. In \autoref{Sec3}, we prove monotonicity formulas for the flows \eqref{GX}, \eqref{GX-p>n+1}. The  uniform $C^k$ estimates for the flows \eqref{GX}, \eqref{GX-p>n+1} are established in \autoref{Sec4}. For the derivation of the $C^2$ estimates along our flows, to circumvent the extra difficulties arising from the boundary, we first derive both lower and upper bounds on the Gauss curvature and then the upper bound on the principal radii of curvature. In \autoref{Sec5}, we provide the proofs of \autoref{MTO} and \autoref{MTO-noneven}.

\section{Preliminaries}
\label{Sec2}
 
Let $\cM$ be a compact orientable smooth manifold of dimension $n$ with boundary $\partial \cM$ and let $X: \cM \to \overline{\mathbb{R}^{n+1}_{+}}$ be a properly embedded smooth hypersurface, i.e., 
\eq{
X(\operatorname{int}\cM) \subset \mathbb{R}^{n+1}_{+}
\quad \text{and} \quad
X(\partial \cM) \subset \partial \mathbb{R}^{n+1}_{+}. 
}
We further assume that $\Sigma = X(\cM)$ is strictly convex, meaning that $\Sigma$ together with $\partial \bbR^{n+1}_+$ bounds a convex body and that $\Sigma$ has a positive definite second fundamental form. Set $\partial \Sigma = X(\partial \cM)$. 
Note that $\partial \Sigma$ is a closed, strictly convex hypersurface in $\partial \mathbb{R}^{n+1}_{+}$ (see, e.g., \cite[Cor. 2.5]{WWX24}). Write $\widehat{\partial\Sigma}$ for the enclosed region by $\partial \Sigma$ in $\partial \mathbb{R}^{n+1}_{+}$. 

For any $r > 0$, the capillary spherical cap of radius $r$ intersecting the support hyperplane $\partial \mathbb{R}^{n+1}_{+}$ at a constant angle $\theta \in (0,\pi)$ is defined by
\eq{ \label{s2:def-C_theta}
\mathcal{C}_{\theta,r} = \{\zeta \in \overline{\mathbb{R}^{n+1}_{+}} : |\zeta - r \cos\theta \, e| = r \}.
}

Let us take $\cM = \cC_\theta = \cC_{\theta,1}$ and $X=\tilde{\nu}^{-1}$. Given a strictly convex capillary hypersurface $\Sigma$, the capillary support function of $\Sigma$, $h = h_{\Sigma} : \mathcal{C}_\theta \to \mathbb{R}$, 
is defined by
\eq{ \label{s2:def-capillary-spt}
h(\xi) := X(\xi) \cdot \nu\big(X(\xi)\big) = \tilde{\nu}^{-1}(\xi) \cdot (\xi - \cos \theta\, e), \quad \forall \, \xi \in \mathcal{C}_\theta.
}
Then $h$ satisfies the capillary boundary condition (cf. \cite[Lem. 2.4]{MWWX25})
\eq{
\nabla_{\mu} h = \cot \theta\, h, \quad \text{on } \partial \mathcal{C}_\theta.
}
When $\Sigma = \cC_\theta$, we have $X(\xi) = \xi$ for all $\xi \in \cC_\theta$ and
\eq{ \label{s2:capillary-spt-fucntion-C}
h_{\cC_\theta}(\xi) &= \xi \cdot (\xi - \cos \theta\, e) = 1 + \cos\theta\, (\nu \cdot e) = \ell(\xi).
}

We write
$
\hat{h} = \hat{h}_{\widehat{\Sigma}} : \mathbb{S}^n \to \mathbb{R}
$
for the (standard) support function of the convex body $\widehat{\Sigma}$, defined by
\eq{
\hat{h}(u) := \max_{x \in \widehat{\Sigma}} u \cdot x, \quad \forall\, u \in \mathbb{S}^n.
}
In particular, we have
\eq{\label{s2:relation-capillary-suppa}
h(\xi)=\hat{h}(\xi-\cos\theta\, e), \quad \forall \xi \in \cC_\theta.
}

In a local orthonormal frame $\{ e_i \}_{i=1}^n$ on $\mathcal{C}_\theta$ (with respect to the standard metric of $\mathcal{C}_\theta$), the second fundamental form $b_{ij}$ of $\Sigma$ is given by
\eq{
  b_{ij} = \nabla_{ij}^2 h + h \, \delta_{ij},
}
and the Gauss curvature of $\Sigma$ at $X(\xi)$ is
\eq{
  \mathcal{K}\big( X(\xi) \big) = \frac{1}{\det\big( \nabla^{2} h(\xi) + h(\xi) I \big)}.
}
It is often more convenient to consider $\mathcal{K}$ as a function on $\mathcal{C}_{\theta}$.

For our later purpose, set
\eq{\label{s2:def-S-theta}
\mathbb{S}_{\theta} = \{ y \in \mathbb{S}^n : y_{n+1} \geq \cos \theta \}.
}
We define the map $T: \mathbb{S}_{\theta} \to \mathcal{C}_\theta$ by
\eq{\label{s2:def-T}
T(u) = u + \cos \theta\, e.
}

\section{Monotone Quantities}
\label{Sec3}
Recall that $\tilde{\nu} = \nu + \cos\theta\, e$,
and hence
\eq{\label{VX}
\tilde{\nu} \cdot \nu = 1 + \cos\theta\, (\nu \cdot e).
}
Let $\Sigma_{\tau}$ be the solution to \eqref{GX}.
Then the capillary support function $h(\cdot,\tau)$ of $\Sigma_{\tau}$ satisfies the following equations:
\eq{\label{Gh}
\left\{
\begin{aligned}
\partial_{\tau} h &= - \frac{(n+1)V(\widehat{\Sigma_0})}{\int_{\mathcal{C}_\theta} f h^{p}\, d\xi} f h^{p} \cK + h, & \quad \text{in } \mathcal{C}_\theta \times [0,\infty), \\
\nabla_{\mu} h &= \cot\theta\, h, & \quad \text{on } \partial \mathcal{C}_\theta \times [0,\infty), \\
h(\xi,0) &= h_0(\xi), & \quad \text{in } \mathcal{C}_\theta.
\end{aligned}
\right.
}

For the flow \eqref{Gh}, we define the functional
\eq{
J(\tau) = \left\{
\begin{aligned}
&-\frac{1}{n+1} \log V(\widehat{\Sigma_0})
+ \frac{1}{p} \log \int_{\mathcal{C}_\theta} f(\xi) h(\xi,\tau)^{p}\, d\xi, & \quad p \neq 0, \\
&-\frac{1}{n+1} \log V(\widehat{\Sigma_0}) + \frac{\int_{\mathcal{C}_\theta} f(\xi) \log h(\xi,\tau)\, d\xi}{\int_{\mathcal{C}_\theta} f(\xi)\, d\xi}, & \quad p = 0.
\end{aligned}
\right.
}
We show that $J(\tau)$ is non-increasing.

\begin{lemma}\label{monJ2}
The functional $J(\tau)$ is non-increasing along the flow \eqref{GX} and equality holds if and only if $\Sigma_{\tau}$ satisfies \eqref{s1:static-eq}.
\end{lemma}

\begin{proof}For $\tau > 0$, the derivative $\partial_{\tau} h$ is a capillary function.  
By definition, $V(\widehat{\Sigma_{\tau}})$ is constant in time.  
Hence, by \cite[Prop. 2.9]{MWWX25}, we obtain
\begin{equation} \label{Var-formula}
0 = \frac{d}{d\tau} V(\widehat{\Sigma_\tau}) 
  = \int_{\mathcal{C}_\theta} \frac{\partial_\tau h}{\mathcal{K}} \, d\xi.
\end{equation}
For $p \neq 0$, we have
\eq{
\frac{d}{d\tau} J(\tau) &= -\int_{\mathcal{C}_\theta} \frac{\partial_\tau h}{(n+1)V(\widehat{\Sigma_{0}}) \cK} \, d\xi + \frac{\int_{\mathcal{C}_\theta} f h^{p-1}\partial_{\tau} h \, d\xi}{\int_{\mathcal{C}_\theta} f h^{p} \, d\xi} \\
&= -\int_{\mathcal{C}_\theta} \frac{\left( -h + \frac{(n+1) V(\widehat{\Sigma_{0}})f h^p \cK}{\int_{\mathcal{C}_\theta} f h^{p} \, d\xi} \right)^{2}}{(n+1) V(\widehat{\Sigma_{0}}) h \cK} \, d\xi \leq 0.
}
Similarly, for $p = 0$, we have
\eq{
\frac{d}{d\tau} J(\tau) &= -\int_{\mathcal{C}_\theta} \frac{\partial_{\tau} h}{(n+1)V(\widehat{\Sigma_{0}}) \cK} \, d\xi + \frac{\int_{\mathcal{C}_\theta} \frac{f}{h}\partial_{\tau} h \, d\xi}{\int_{\mathcal{C}_\theta} f \, d\xi} \\
&= -\int_{\mathcal{C}_\theta} \frac{\left( -h + \frac{(n+1)V(\widehat{\Sigma_{0}}) f \cK}{\int_{\mathcal{C}_\theta} f \, d\xi} \right)^{2}}{(n+1)V(\widehat{\Sigma_{0}}) h \cK} \, d\xi \leq 0.
}
For each $p$, $\frac{d}{d\tau}J(\tau) = 0$ if and only if
\eq{
h^{1-p}\det(\nabla^{2}h + h I) = \frac{(n+1)V(\widehat{\Sigma_{0}})f}{\int_{\cC_\theta}f h^p \, d\xi}.
}
\end{proof}

Now we turn to the flow \eqref{GX-p>n+1}. Along this flow, the capillary support function $h(\cdot,\tau)$ of $\Sigma_{\tau}$ solves
\eq{ \label{h-flow 2}
\left\{ \begin{aligned}
\partial_{\tau} h &= - f h^p \cK + h, &\quad \text{in $\cC_\theta \times [0,\infty)$},\\
\nabla_\mu h &= \cot\theta\, h, &\quad \text{on $\partial \cC_\theta \times [0,\infty)$}, \\
h(\cdot,0) &= h_0(\cdot), &\quad \text{in $\cC_\theta$}.
\end{aligned}\right.
}
We define the functional
\eq{
\widetilde{J}(\tau) = -V(\widehat{\Sigma_{\tau}}) + \frac{1}{p}\int_{\cC_\theta}f(\xi)h(\xi,\tau)^p d\xi,
}
and we show that $\widetilde{J}(\tau)$ is non-increasing along the flow \eqref{GX-p>n+1}.

\begin{lemma}\label{monJ-noneven}
The functional $\widetilde{J}(\tau)$ is non-increasing along the flow \eqref{GX-p>n+1} and equality holds if and only if $\Sigma_{\tau}$ satisfies \eqref{Lp-MP-smooth}.
\end{lemma}

\begin{proof}
We have
\eq{
\frac{d}{d\tau} \widetilde{J}(\tau) &= -\int_{\mathcal{C}_\theta} \frac{\partial_{\tau} h}{ \cK } \, d\xi + \int_{\mathcal{C}_\theta} f h^{p-1}\partial_{\tau} h \, d\xi \\
&= -\int_{\mathcal{C}_\theta} \frac{\left( -h + f h^p \cK \right)^{2}}{ h \cK} \, d\xi \leq 0.
}
It is clear that $\frac{d}{d\tau}\widetilde J(\tau) = 0$ if and only if $ h^{1-p}\det(\nabla^{2}h + h I) = f$.
\end{proof}

\section{Regularity Estimates}\label{Sec4}
\subsection{Regularity Estimates I}\label{Sec4-1}
In this subsection, we deduce the a priori estimates for solutions to the flow \eqref{GX}. First, we establish both the uniform lower and upper bounds of the capillary support function.

\begin{lemma}\label{s3:C0-estimate}
Let $p > -n-1$ and $\theta \in (0,\frac{\pi}{2})$. Suppose $\Sigma_{\tau}$ is a smooth, even, and strictly convex solution to the flow \eqref{GX}. Then there exists a positive constant $C$, depending only on $n, p$, $\theta, \min_{\mathcal{C}_\theta}f, \max_{\mathcal{C}_\theta}f$ and $\Sigma_{0}$, such that
\eq{\label{s3:C0-estimate*}
\frac{1}{C} \leq h(\xi,\tau) \leq C, \quad \forall (\xi,\tau) \in \mathcal{C}_\theta \times (0,\infty).
}
\end{lemma}
\begin{proof}
\noindent\emph{Case 1:} $-(n+1) < p < 0$. From \autoref{monJ2}, we have
\eq{
\log V(\widehat{\Sigma_0}) - \frac{n+1}{p} \log \int_{\mathcal{C}_\theta} f(\xi) h(\xi,\tau)^{p} \, d\xi \geq -(n+1) J(0).
}
It follows that
\eq{\label{bb}
V(\widehat{\Sigma_{\tau}}) \left( \int_{\mathcal{C}_\theta} f(\xi) h(\xi,\tau)^{p} \, d\xi \right)^{-\frac{n+1}{p}} \geq e^{-(n+1) J(0)}.
}
Let $\Xi_{\tau}$ be the convex body enclosed by $\Sigma_{\tau}$ and its reflection across $\partial \mathbb{R}^{n+1}_{+}$. Define the function $\tilde{f} : \mathbb{S}^n \to (0,\infty)$ by
\eq{
\tilde{f}(u) = \left\{
\begin{array}{l@{\ }l}
f \circ T(u), & \quad u_{n+1} \geq \cos \theta, \\
f \circ T(u_1, \ldots, u_n, -u_{n+1}), & \quad u_{n+1} \leq -\cos \theta, \\
f \circ T(\sin \theta\, \tilde{u} + \cos \theta\, E_{n+1}), & \quad -\cos \theta \leq u_{n+1} \leq \cos \theta,
\end{array}
\right.
}
where $\tilde{u} := \dfrac{u - (u \cdot e)e}{|u - (u \cdot e)e|} \in \mathbb{S}^{n} \cap e^\bot$ and the map $T$ is given by \eqref{s2:def-T}. Note that $\tilde{f}$ is the symmetric extension of $f$ from $\mathcal{C}_\theta$ to the whole $\mathbb{S}^n$. Note that $\hat{h}_{\Xi_{\tau}}(u) = h_{\Sigma_{\tau}} \circ T(u)$ for $u \in \mathbb{S}_{\theta}$, where $\mathbb{S}_\theta$ is defined in \eqref{s2:def-S-theta}. By \eqref{bb} we have
\eq{\label{b2}
V(\Xi_{\tau}) \left( \int_{\mathbb{S}^n} \tilde{f} \hat{h}^{p}_{\Xi_{\tau}} \, du \right)^{-\frac{n+1}{p}} \geq e^{-(n+1) J(0)}.
}
In view of \cite[Lem. 2.6]{HI25}, the capillary support function is uniformly bounded above and below away from zero.

For the remaining cases, with the volume fixed, it is sufficient to prove the uniform upper bound on the capillary support function. Given that the origin lies in the interior of the flat side $\widehat{\partial \Sigma_{\tau}} \subset \partial \mathbb{R}^{n+1}_{+}$, there exist $R_{\tau} > 0$ and $q_{\tau} \in \overline{\mathbb{R}_+^{n+1}}$ such that
\eq{
\label{CR}
\frac{1}{R_{\tau}}\widehat{\Sigma_{\tau}} \subseteq \widehat{\mathcal{C}_\theta}, \quad q_{\tau} \in \frac{1}{R_{\tau}}\Sigma_{\tau} \cap \mathcal{C}_\theta.
}
Rotating the coordinates around the $E_{n+1}$-axis, one may assume that
\eq{
q_{\tau} = (q_1, 0, \ldots, 0, q_{n+1}) \in \mathcal{C}_\theta, \quad 0 \leq q_{n+1} \leq 1 - \cos\theta,
}
where $q_1 = \sqrt{1 - (q_{n+1} + \cos\theta)^2} \geq 0$. We have
\eq{
|q_{\tau}|^2 = q_1^2 + q_{n+1}^2 = 1 - \cos^2\theta - 2\cos\theta\, q_{n+1} \geq (1 - \cos\theta)^2.
}

By the evenness of the hypersurface $\Sigma_{\tau}$, we also have
\eq{
  R_{\tau}(-q_1, 0, \ldots, 0, q_{n+1}) \in \Sigma_{\tau}.
}
Therefore, by \eqref{s2:relation-capillary-suppa}, we obtain
\eq{
  h(\xi, \tau) &\geq R_{\tau} (q_1, 0, \ldots, 0, q_{n+1}) \cdot (\xi_1, \ldots, \xi_n, \xi_{n+1} + \cos\theta) \\
  &= R_{\tau} \big( \xi_1 q_1 + (\xi_{n+1} + \cos\theta) q_{n+1} \big),
}
as well as
\eq{
  h(\xi, \tau) \geq R_{\tau} \big( -\xi_1 q_1 + (\xi_{n+1} + \cos\theta) q_{n+1} \big).
}
Since $q_{n+1} \geq 0$ and $\xi_{n+1} + \cos\theta \geq 0$, we obtain
\eq{
h(\xi, \tau) \geq R_{\tau} |\xi_1| q_1, \quad h(\xi, \tau) \geq R_{\tau} (\xi_{n+1} + \cos\theta) q_{n+1}.
}
Therefore, for $p_{\tau} := q_{\tau}/|q_{\tau}| \in \mathbb{S}^n$, we have
\eq{\label{ht}
h(\xi, \tau) &\geq \frac{R_{\tau}}{2} \big| q_{\tau} \cdot (\xi_1, \ldots, \xi_n, \xi_{n+1} + \cos\theta) \big| \\
&\geq C_{\theta} R_{\tau} \big| p_{\tau} \cdot (\xi+\cos\theta E_{n+1}) \big|,
}
where $C_{\theta} := \frac{1}{2} (1 - \cos\theta)$.

\noindent\emph{Case 2:} $p = 0$. 
By \autoref{monJ2}, we have
\eq{\label{aa}
J(0) + \frac{1}{n+1} \log V(\widehat{\Sigma_0}) \geq \frac{\int_{\mathcal{C}_\theta} f(\xi) \log h(\xi,\tau) \, d\xi}{\int_{\mathcal{C}_\theta} f(\xi) \, d\xi}.
}
Substituting \eqref{ht} into \eqref{aa}, we have
\eq{\label{a2}
&~(\log (C_{\theta}R_{\tau})) \int_{\mathcal{C}_\theta} f(\xi) \, d\xi + (\max_{\cC_{\theta}} f)\int_{\bbS^n} \log |u \cdot p_{\tau}|\, du\\
\leq &~(\log (C_{\theta}R_{\tau})) \int_{\mathcal{C}_\theta} f(\xi) \, d\xi + \int_{\bbS_\theta} f(u+\cos\theta\,e) \log |u \cdot p_{\tau}| \, du\\
\leq &~\int_{\mathcal{C}_\theta} f(\xi) \log h(\xi,\tau) \, d\xi
\leq \left( J(0) + \frac{1}{n+1} \log V(\widehat{\Sigma_0}) \right)\int_{\mathcal{C}_\theta} f(\xi) \, d\xi.
}
Since the integral $\int_{\bbS^n} \log |u \cdot p_{\tau}|\, du$ does not depend on $p_{\tau}$, for some positive constant $C$, depending only on $n, p$, $\theta, \min_{\mathcal{C}_\theta}f, \max_{\mathcal{C}_\theta}f$ and $\Sigma_{0}$,  we obtain
\eq{
R_{\tau}\leq C.
}
\emph{Case 3:} $p > 0$. Using \autoref{monJ2} and \eqref{ht}, we have
\eq{
V(\widehat{\Sigma_0})^{\frac{p}{n+1}} e^{p J(0)} &\geq \int_{\mathcal{C}_\theta} f(\xi) h(\xi,\tau)^{p} \, d\xi \\
&\geq C_{\theta}^pR_{\tau}^{p} (\min_{\mathcal{C}_\theta} f) \int_{\bbS_{\theta}} |u\cdot p_{\tau}|^p \, du.
}
Note that $w: \bbS^n\to \bbR$ defined as $w(x)=\int_{\bbS_{\theta}} |u\cdot x|^p \, du$ is a positive continuous function. Hence, $w\geq c_{\theta}>0$. 
Thus we obtain
\eq{
R_{\tau} &\leq \frac{V(\widehat{\Sigma_0})^{\frac{1}{n+1}} e^{J(0)}}{c_{\theta}^\frac{1}{p}C_{\theta}(\min_{\mathcal{C}_\theta} f)^\frac{1}{p}}.
}
\end{proof}

\begin{lemma}\label{C1}
Let $p > -n-1$ and $\theta \in (0,\frac{\pi}{2})$. Suppose $\Sigma_{\tau}$ is a smooth, even, and strictly convex solution to the flow \eqref{GX}. Then there exists a positive constant $C$, depending only on $n, p$, $\theta, \min_{\mathcal{C}_\theta}f, \max_{\mathcal{C}_\theta}f$ and $\Sigma_{0}$, such that
\eq{
|\nabla h(\cdot,\tau)| \leq C.
}
\end{lemma}
\begin{proof}
By \autoref{s3:C0-estimate}, $h \leq C$. The claim follows from \cite[Lem. 4.8]{HIS25}.
\end{proof}

Next, we derive a uniform upper bound for the Gauss curvature $\cK(\cdot, \tau)$.
\begin{lemma}\label{lemma-Gauss-curv}
Let $p > -n-1$ and $\theta \in (0,\frac{\pi}{2})$. Suppose $\Sigma_{\tau}$ is a smooth, even, and strictly convex solution to the flow \eqref{GX}. Then there exists a  positive constant $C$, depending only on $n, p$, $\theta, \min_{\mathcal{C}_\theta}f, \max_{\mathcal{C}_\theta}f$ and $\Sigma_0$, such that
\eq{
\cK \leq C.
}
\end{lemma}
\begin{proof}
We apply the maximum principle to the function
\eq{\label{AF1}
Q = \frac{ \frac{(n+1)V(\widehat{\Sigma_{0}})}{\int_{\mathcal{C}_\theta} f h^{p} \, d\xi}f h^{p} \cK - h}{h - \varepsilon_0} = \frac{-h_{\tau}}{h - \varepsilon_0},
}
where $h_{\tau} :=\partial_{\tau} h$ and
\eq{\label{AF2}
\varepsilon_0 := \frac{1}{2} \min_{\mathcal{C}_\theta \times [0,\infty)} h(\xi,\tau) > 0.
}
We may assume that $\max \cK \gg 1$. Due to the $C^0$ estimate (cf.  \autoref{s3:C0-estimate}), we may also assume that $\max Q\approx\max \cK  \gg 1$.

For any fixed $\tau \in (0,\infty)$, suppose the maximum of $Q(\xi,\tau)$ is attained at some point $\xi_0 \in \mathcal{C}_\theta$. On $\partial \mathcal{C}_\theta$, we have
\eq{
\nabla_\mu Q &= -\frac{h_{\tau\mu}}{h - \varepsilon_0} + \frac{h_{\tau} h_{\mu}}{(h - \varepsilon_0)^2} = \frac{\varepsilon_0 \cot\theta\, h_{\tau}}{(h - \varepsilon_0)^2} = -\frac{\varepsilon_0 \cot\theta\, Q}{h - \varepsilon_0} < 0.
}
Therefore, the maximum of $Q(\cdot,\tau)$ is attained at some point $\xi_0 \in \mathcal{C}_\theta \setminus \partial \mathcal{C}_\theta$. 

Choose a local orthonormal frame $\{e_i\}_{i=1}^n$ such that $\{b_{ij}\}$ is diagonal at the point $\xi_0$. Then, at the point $\xi_0$, we have
\eq{\label{Up1}
0 = \nabla_i Q &= \frac{-h_{\tau i}}{h - \varepsilon_0} + \frac{h_{\tau} h_i}{(h - \varepsilon_0)^2},\\
0 \geq \nabla^2_{ii} Q &= \frac{-h_{\tau ii}}{h - \varepsilon_0} + \frac{2 h_{\tau i} h_i + h_{\tau} h_{ii}}{(h - \varepsilon_0)^2} - \frac{2 h_{\tau} h_i^2}{(h - \varepsilon_0)^3} \\
&= \frac{-h_{\tau ii}}{h - \varepsilon_0} + \frac{h_{\tau} h_{ii}}{(h - \varepsilon_0)^2}.
}
This implies that
\eq{
\label{Up3}
-h_{\tau ii} - h_{\tau} &\leq -\frac{h_{\tau} h_{ii}}{h - \varepsilon_0} - h_{\tau} = \frac{-h_{\tau}}{h - \varepsilon_0} [h_{ii} + (h - \varepsilon_0)] = Q (b_{ii} - \varepsilon_0).
}
Moreover, we have
\eq{
\label{Up4*}
\partial_{\tau} Q &= \frac{-h_{\tau\tau}}{h - \varepsilon_0} + \frac{h_{\tau}^2}{(h - \varepsilon_0)^2} \\
&= \frac{(n+1)V(\widehat{\Sigma_{0}}) f h^{p} \cK}{(h - \varepsilon_0) \int_{\mathcal{C}_\theta} f h^p \, d\xi} \left( \frac{\cK_\tau}{\cK} - \frac{p \int_{\mathcal{C}_\theta} f h^{p-1} h_{\tau} \, d\xi}{\int_{\mathcal{C}_\theta} f h^{p} \, d\xi} + p \frac{h_{\tau}}{h} \right) + Q + Q^2.
}
For simplicity, let us put 
\eq{
\sigma_n = \det(\nabla^2 h + h I), \quad \sigma_n^{ij} = \frac{\partial \det(\nabla^2 h + h I)}{\partial b_{ij}}.
}
Using \eqref{Up3}, at the point $\xi_0$, we calculate that
\eq{
\label{Up5}
\cK_{\tau} = \partial_{\tau} \sigma_n^{-1}
&= -\sigma_n^{-2} \sum_{i,j} \sigma_n^{ij} (h_{\tau ij} + h_{\tau} \delta_{ij}) \\
&\leq \sigma_n^{-2} \sum_i \sigma_n^{ii} Q (b_{ii} - \varepsilon_0) \\
&= \cK Q (n - \varepsilon_0 \sum_i b^{ii}),
}
where 
$\{b^{ij}\}$ denotes the inverse matrix of $\{b_{ij}\}$. Here, we used $\sigma_n^{ii} = \sigma_n b^{ii}$ and 
\eq{
\sum_i \sigma_n^{ii} b_{ii} &= n \sigma_n, \quad
\sum_i \sigma_n^{ii} &= \sigma_n \sum_i b^{ii}.
}
It follows from \eqref{Up5} and 
$\sum_i b^{ii} \geq n (\prod_i b^{ii})^{\frac{1}{n}} = n \cK^{\frac{1}{n}}$ that 
\eq{
\label{Up66}
\frac{\cK_{\tau}}{\cK} \leq n Q (1 - \varepsilon_0 \cK^{\frac{1}{n}}).
}
Substituting \eqref{Up66} and \eqref{AF1} into \eqref{Up4*}, we obtain
\begin{align*}
\partial_{\tau} Q &\leq \frac{(n+1)V(\widehat{\Sigma_{0}}) f h^p \cK}{(h - \varepsilon_0) \int_{\mathcal{C}_\theta} f h^p \, d\xi} \Big( n Q (1 - \varepsilon_0 \cK^{\frac{1}{n}}) \\
&\quad + \frac{p \int_{\mathcal{C}_\theta} f h^{p-1} (h - \varepsilon_0) Q \, d\xi}{\int_{\mathcal{C}_\theta} f h^{p} \, d\xi} - \frac{p (h - \varepsilon_0) Q}{h} \Big)  + Q + Q^2 \\
&= \big( Q + \frac{h}{h - \varepsilon_0} \big) \Big( n Q (1 - \varepsilon_0 \cK^{\frac{1}{n}}) \\
&\quad + \frac{p \int_{\mathcal{C}_\theta} f h^{p-1} (h - \varepsilon_0) Q \, d\xi}{\int_{\mathcal{C}_\theta} f h^{p} \, d\xi} - \frac{p (h - \varepsilon_0) Q}{h} \Big)  + Q + Q^2.
\end{align*}
Hence, we can find $C_1$ independent of $\tau$ such that  
\eq{
Q(\cdot,\tau) \leq C_1.
}
Therefore,  
$
\cK(\cdot,\tau) \leq C_2
$
for some $C_2$ depending only on $n, p$, $\theta, \min_{\mathcal{C}_\theta}f, \max_{\mathcal{C}_\theta}f$ and $\Sigma_0$.
\end{proof}

We now show that the Gauss curvature $\cK(\cdot,\tau)$ is uniformly bounded below away from zero.

\begin{lemma}\label{lemma-Gauss-curv-lower}
Let $p > -n-1$ and $\theta \in (0,\frac{\pi}{2})$. Suppose $\Sigma_{\tau}$ is a smooth, even, and strictly convex solution to the flow \eqref{GX}. Then there exists a positive constant $c$, depending only on $n, p$, $\theta$, $\min_{\mathcal{C}_\theta}f, \max_{\mathcal{C}_\theta}f$ and $\Sigma_0$, such that
\[
\cK \geq c.
\]
\end{lemma}

\begin{proof}
We apply the maximum principle to the auxiliary function
\eq{
Q = \log (f^{-1} \cK^{-1}) - N \log h,
}
where $N > 0$ is a constant to be determined later. We may assume that $\min \cK \ll 1$. Due to the $C^0$ estimate of $h$ (cf. \autoref{s3:C0-estimate}), we may also assume that $\max Q \gg 1$.

\begin{claim}
For $\tau > 0$, we have $\nabla_\mu Q < 0$ on $\partial \mathcal{C}_\theta$, provided $N > p-1$.
\end{claim}

For simplicity, set $\al := \frac{(n+1)V(\widehat{\Sigma_{0}})}{\int_{\mathcal{C}_\theta} f h^{p} \, d\xi}$. For $\tau>0$, we calculate
\eq{
\label{htu2}
h_{\tau\mu} &= \nabla_{\mu} (-\al f h^{p} \cK + h) \\
&= -\al f_{\mu} \cK h^{p} - \al f \cK_{\mu} h^{p} - p \al f \cK h^{p} \cot\theta + \cot\theta\, h,
}
and
\eq{\label{htu3}
h_{\tau\mu} = \cot \theta\, h_{\tau} = \cot \theta \, (-\al f \cK h^{p} + h).
}
From \eqref{htu2} and \eqref{htu3}, for $\tau > 0$, we have
\eq{\label{IJ}
\frac{\cK_{\mu}}{\cK} = -\frac{f_{\mu}}{f} + (1-p) \cot \theta,
}
and hence, if $N > p-1$, then
\eq{
\nabla_\mu Q &= -\frac{f_{\mu}}{f} - \frac{\cK_\mu}{\cK} - N \frac{h_\mu}{h} = (p-1) \cot \theta - N \cot \theta < 0.
}

Therefore, when $N>p-1$, for any fixed $\tau \in (0,\infty)$, the maximum of $Q(\cdot,\tau)$ is attained at some point $\xi_0 \in \mathcal{C}_\theta \setminus \partial \mathcal{C}_\theta$. Now, take a local orthonormal frame $\{e_i\}_{i=1}^n$ around $\xi_0$ such that $\{b_{ij}\}(\xi_0,\tau)$ is diagonal. At the point $\xi_0$,
\eq{\label{Ps3:C0-estimate}
\nabla_i Q = -\frac{\nabla_i (f \cK)}{f \cK} - N \frac{\nabla_i h}{h} = 0,
}
and
\eq{
\label{KIP}
\nabla^2_{ii} Q &= -\frac{\nabla^2_{ii} (f \cK)}{f \cK} + \frac{(\nabla_i (f \cK))^2}{(f \cK)^2} - N \frac{\nabla^2_{ii} h}{h} + N \frac{(\nabla_i h)^2}{h^2} \leq 0.
}
Substituting \eqref{Ps3:C0-estimate} into \eqref{KIP}, we obtain
\eq{\label{PC3}
-\nabla^2_{ii} (f \cK) \leq N f \cK \frac{\nabla^2_{ii} h}{h} + (-N - N^2) f \cK \frac{(\nabla_i h)^2}{h^2}.
}
In addition, we have
\eq{
\label{PCC}
\partial_{\tau} Q &= \frac{\partial_{\tau} \sigma_n}{\sigma_n} - N \frac{h_{\tau}}{h} = \sum_i b^{ii} (h_{\tau ii} + h_{\tau}) - N \frac{h_{\tau}}{h},
}
and
\eq{
\label{HTT}
h_{\tau ii} &= \nabla^2_{ii} (-\al f h^{p} \cK + h) \\
&= -\al h^{p} \nabla^2_{ii} (f \cK) - 2 p \al h^{p-1} \nabla_i h \nabla_i (f \cK) \\
&\quad - p (p-1) \al h^{p-2} f \cK (\nabla_i h)^2 + (1 - p \al h^{p-1} f \cK) \nabla^2_{ii} h.
}
In view of \eqref{Ps3:C0-estimate} and \eqref{PC3}, we find
\eq{\label{LP}
h_{\tau ii} \leq &~(1 +\al (N - p)  h^{p-1} f \cK)  \nabla^2_{ii} h \\
&~+ (-N - N^2 + 2 p N - p (p-1)) \al h^{p-2} f \cK (\nabla_i h)^2  \\
=&~ (1 + (N - p) \al h^{p-1} f \cK) (b_{ii} - h) \\
&~ + (-N - N^2 + 2 p N - p (p-1)) \al h^{p-2} f \cK (\nabla_i h)^2.
}
Substituting \eqref{LP} into \eqref{PCC}, we deduce
\eq{
\partial_{\tau} Q &\leq (1 + \al(N - p)  h^{p-1} f \cK) \big( n - h \sum_i b^{ii} \big) \\
&\quad + (-N - N^2 + 2 p N - p (p-1)) \al h^{p-2} f \cK \sum_i b^{ii} (\nabla_i h)^2 \\
&\quad + (-\al h^p f \cK  + h) \sum_i b^{ii} - N (1 - \al h^{p-1} f \cK ) \\
&= \left( n - N + (n (N - p) + N) \al h^{p-1} f \cK \right) - (N - p + 1) \al h^p f \cK \sum_i b^{ii} \\
&\quad + (-N - N^2 + 2 p N - p (p-1)) \al h^{p-2} f \cK \sum_i b^{ii} (\nabla_i h)^2.
}
Due to the $C^0$ estimate (cf. \autoref{s3:C0-estimate}), if $\max Q \to \infty$, i.e., $\min \cK \to 0$, the right-hand side of the above inequality is strictly negative, provided $N$ is sufficiently large. Thus, the Gauss curvature $\cK$ is uniformly bounded below away from zero.
\end{proof}

Finally, we prove that the principal curvatures of $\Sigma_{\tau}$ are uniformly bounded by positive constants from both above and below.

\begin{lemma}\label{s4:lem-C2*}
Let $p > -n-1$ and $\theta \in (0, \frac{\pi}{2})$. Suppose $\Sigma_{\tau}$ is a smooth, even, and strictly convex solution to the flow \eqref{GX}. Then there exists a positive constant $C$, depending only on $n$, $p$, $\theta$, $\min_{\mathcal{C}_\theta} f$, $\|f\|_{C^2(\mathcal{C}_\theta)}$ and $\Sigma_0$, such that the principal curvatures $\kappa_1, \ldots, \kappa_n$ of $\Sigma_{\tau}$ satisfy
\eq{\label{PUL}
\frac{1}{C} \leq \kappa_i \leq C.
}
\end{lemma}
\begin{proof}
Since the Gauss curvature $\cK$ of $\Sigma_{\tau}$ is uniformly bounded from above (cf. \autoref{lemma-Gauss-curv}), it suffices to prove the lower bound of the principal curvatures of $\Sigma_{\tau}$, which is equivalent to the upper bound of the principal radii of curvature of $\Sigma_{\tau}$. To this end, we apply the maximum principle to the auxiliary function
\eq{
Q= \sigma_1+ \frac{A}{2} |\nabla h|^2,
}
where $A > 0$ is a constant to be specified later, and $\sigma_1 := \sigma_1(b_{ij}(\xi,\tau))$ is the trace of the matrix $\{b_{ij}\}$.

\begin{claim}
Fix $A > 0$ and suppose $\tau > 0$. Then $\nabla_{\mu} Q < 0$ at any point on $\partial \mathcal{C}_{\theta}$ where $Q$ is sufficiently large (depending on $A$).
\end{claim}

Choose a local orthonormal frame $\{e_i\}_{i=1}^n$ such that $e_n = \mu$ and $\{b_{ij}\}$ is diagonal at $\xi_0 \in \partial \mathcal{C}_\theta$. At the point $\xi_0$, we have
\eq{\label{nabla=E}
\nabla_\mu Q &= \sum_j \nabla_\mu b_{jj} + A \sum_j h_j h_{j\mu} \\
&= \cot \theta \sum_{\beta=1}^{n-1} (b_{\mu\mu} - b_{\beta\beta}) + \nabla_\mu b_{\mu\mu} + A \cot \theta\, h h_{\mu\mu},
}
where we used (due to the Gauss-Weingarten equation)
\eq{
h_{\mu\beta} = 0 \quad \text{and} \quad \nabla_{\mu} b_{\beta\beta} = \cot \theta\, (b_{\mu\mu} - b_{\beta\beta})
}
for all $1 \leq \beta \leq n-1$.
On the other hand, we have
\eq{
\frac{\cK_\mu}{\cK} &= -\sigma_n^{-1} \nabla_\mu \sigma_n \\
&= -\sigma_n^{-1} \left( \sum_{\beta} \sigma_n^{\beta\beta} \nabla_\mu b_{\beta\beta} + \sigma_n^{\mu\mu} \nabla_\mu b_{\mu\mu} \right) \\
&= -\left( \cot \theta \sum_{\beta} b^{\beta\beta} (b_{\mu\mu} - b_{\beta\beta}) + b^{\mu\mu} \nabla_\mu b_{\mu\mu} \right) \\
&= -\cot \theta\, b_{\mu\mu} \sum_{\beta} b^{\beta\beta} + (n-1) \cot \theta - b^{\mu\mu} \nabla_\mu b_{\mu\mu}.
}
Combining this with \eqref{IJ}, we obtain
\eq{\label{eq-nabla-b_nn}
\nabla_\mu b_{\mu\mu} = \left( \frac{f_\mu}{f} + (n+p-2) \cot \theta \right) b_{\mu\mu} - \cot \theta\, b_{\mu\mu}^2 \sum_{\beta} b^{\beta\beta}.
}
Substituting \eqref{eq-nabla-b_nn} into \eqref{nabla=E}, we find
\eq{
\nabla_\mu Q &= \cot \theta \left( n b_{\mu\mu} - \sigma_1 \right) + \left( \frac{f_\mu}{f} + (n+p-2) \cot \theta \right) b_{\mu\mu} \\
&\quad - \cot \theta\, b_{\mu\mu}^2 \sum_{\beta} b^{\beta\beta} + A \cot \theta\, h (b_{\mu\mu} - h) \\
&\leq -\cot \theta\, \sigma_1 + \left( \frac{f_\mu}{f} + (2n + p - 2 + A h) \cot \theta \right) b_{\mu\mu} \\
&\quad - (n-1) \cot \theta\, b_{\mu\mu}^{2 + \frac{1}{n-1}} \cK^{\frac{1}{n-1}},
}
where we used
\eq{
\sum_{\beta} b^{\beta\beta} \geq (n-1) \left( \prod_{\beta} b^{\beta\beta} \right)^{\frac{1}{n-1}} = (n-1) \cK^{\frac{1}{n-1}} b_{\mu\mu}^{\frac{1}{n-1}}.
}
In view of  \autoref{s3:C0-estimate} and the lower bound on the Gauss curvature (cf. \autoref{lemma-Gauss-curv-lower}), we have
\eq{
\nabla_\mu Q \leq \cot \theta \left( -\sigma_1 + c_1 b_{\mu\mu} - c_2 b_{\mu\mu}^{2 + \frac{1}{n-1}} \right)
}
for some positive constants $c_1 = c_1(n,p,\theta,\min_{\mathcal{C}_{\theta}}f,\|f\|_{C^1(\mathcal{C}_\theta)},\Sigma_0,A)$, and $c_2=c_2(n,p,\theta,\min_{\mathcal{C}_{\theta}}f,\max_{\mathcal{C}_{\theta}}f,\Sigma_0)$, independent of $\tau$.

If $Q \geq c_1 b_{\mu\mu}+\frac{A}{2}|\nabla h|^2$ (i.e. $\sigma_1 \geq c_1 b_{\mu\mu}$) at $\xi_0$, then it is clear that $\nabla_\mu Q < 0$. Otherwise, $Q \leq c_1 b_{\mu\mu}+\frac{A}{2}|\nabla h|^2$ (i.e. $\sigma_1 \leq c_1 b_{\mu\mu}$) at $\xi_0$, and hence
\eq{
\nabla_\mu Q &\leq \cot \theta\, b_{\mu\mu} \left( c_1 - c_2 b_{\mu\mu}^{\frac{n}{n-1}} \right) \\
&\leq \cot \theta\, b_{\mu\mu} \left( c_1 - c'_2 \sigma_1^{\frac{n}{n-1}} \right)\\
&= \cot \theta\, b_{\mu\mu} \left( c_1 - c'_2\left(Q-\frac{A}{2}|\nabla h|^2\right)^{\frac{n}{n-1}} \right),
}
which is negative, provided that $Q$ is sufficiently large.

Therefore, for any fixed $\tau \in (0,\infty)$ and any fixed $A$, the maximum of $Q$, if very large, must be attained at some point $\xi_0 \in \mathcal{C}_\theta \setminus \partial \mathcal{C}_\theta$. Choose a local orthonormal frame $\{e_i\}_{i=1}^n$ around $\xi_0$ such that $\{b_{ij}\}(\xi_0,\tau)$ is diagonal, and hence $\{h_{ij}\}(\xi_0,\tau)$ is also diagonal. At $\xi_0$, we calculate
\eq{
\label{XxX1}
0 = \nabla_i Q = \sum_j \nabla_i b_{jj} + A h_i h_{ii},
}
and
\eq{\label{XxX2}
0 \geq \nabla^2_{ii} Q = \sum_j \nabla^2_{ii} b_{jj} + A \left( \sum_j h_j h_{jii} + h_{ii}^2 \right),
}
where we stipulate $h_{jii}:=\nabla_ih_{ji}$.
Moreover, we have
\eq{
\label{XxX3}
\partial_{\tau} Q = \sum_j \partial_{\tau} b_{jj} + A \sum_j h_j h_{j\tau }.
}
On the other hand,
\eq{
\label{XxX4}
\log (h - h_{\tau}) &= \log \left( \al f \cK h^{p} \right) \\
&= -\log \det (\nabla^2 h + h I) + \psi,
}
where $\psi := \log \left(\al f h^p \right)$. Taking the covariant derivative of \eqref{XxX4} in the $e_j$ direction, it follows that
\eq{\label{XxX6}
\frac{h_j - h_{j\tau }}{h - h_{\tau}} &= -\sum_{i,k} b^{ik} \nabla_j b_{ik} + \psi_j \\
&= -\sum_i b^{ii} (h_{jii} + h_i \delta_{ij}) +  \psi_j,
}
where we used that $\nabla b$ is fully symmetric. We also have
\eq{
\label{XxX7}
\frac{h_{jj} - h_{jj\tau }}{h - h_{\tau}} - \frac{(h_j - h_{j\tau })^2}{(h - h_{\tau})^2} &= -\sum_i b^{ii} \nabla^2_{jj} b_{ii} + \sum_{i,k} b^{ii} b^{kk} (\nabla_j b_{ik})^2 +  \psi_{jj}.
}
Using the commutator formula on $\mathbb{S}^n$ we find
\eq{
\label{XxX72}
\frac{h_{jj} - h_{jj\tau }}{h - h_{\tau}} &= \frac{(h_j - h_{j\tau })^2}{(h - h_{\tau})^2} - \sum_i b^{ii} (\nabla^2_{ii} b_{jj} - b_{jj} + b_{ii}) \\
&\quad + \sum_{i,k} b^{ii} b^{kk} (\nabla_j b_{ik})^2 +  \psi_{jj} \\
&= \frac{(h_j - h_{j\tau })^2}{(h - h_{\tau})^2} - \sum_i b^{ii} \nabla^2_{ii} b_{jj} + b_{jj} \sum_i b^{ii} - n \\
&\quad + \sum_{i,k} b^{ii} b^{kk} (\nabla_j b_{ik})^2 +  \psi_{jj}.
}
Together with \eqref{XxX2}, \eqref{XxX3}, and \eqref{XxX72}, we calculate
\eq{
\label{XxX82}
\frac{\partial_{\tau} Q}{h - h_{\tau}} &= \sum_j \frac{(h_{jj\tau } - h_{jj}) + (h_{jj} + h) - h + h_{\tau}}{h - h_{\tau}} + A \frac{\sum_j h_j h_{j\tau }}{h - h_{\tau}} \\
&= -\sum_j \frac{(h_j - h_{j\tau })^2}{(h - h_{\tau})^2} + \sum_{i,j} b^{ii} \nabla^2_{ii} b_{jj} + n^2 - \sum_{i,j,k} b^{ii} b^{kk} (\nabla_j b_{ik})^2 \\
&\quad - \Delta \psi - \sum_j b_{jj} \sum_i b^{ii} + \frac{\sum_j b_{jj}}{h - h_{\tau}} - n + A \frac{\sum_j h_j h_{j\tau }}{h - h_{\tau}} \\
&\leq -A \sum_{i,j} b^{ii} h_j h_{jii} - A \sum_i b^{ii} h_{ii}^2 + n(n-1) - \Delta \psi \\
&\quad - \sigma_1 \sum_i b^{ii}+ \frac{\sigma_1}{h - h_{\tau}} + A \frac{\sum_j h_j h_{j\tau }}{h - h_{\tau}} .
}
Multiplying \eqref{XxX6} by $h_j$ and summing over $j$, we find
\eq{
\frac{|\nabla h|^2 - \sum_j  h_jh_{j\tau }}{h - h_{\tau}} = -\sum_{i,j} b^{ii} h_j h_{jii}  - \sum_i b^{ii} h_i^2 + \sum_j h_j \psi_j.
}
Substituting this into \eqref{XxX82} yields
\eq{
\label{Eht}
\frac{\partial_{\tau} Q}{h - h_{\tau}} &\leq n(n-1) - \Delta \psi - \sigma_1 \sum_i b^{ii} + \frac{\sigma_1}{h - h_{\tau}} \\
&\quad + A \left( \frac{|\nabla h|^2}{h - h_{\tau}} - \sum_j h_j \psi_j + \sum_i b^{ii} (h_i^2 - h_{ii}^2) \right) \\
&= n(n-1) - \Delta \psi  - \sigma_1 \sum_i b^{ii} + \frac{\sigma_1}{h - h_{\tau}}\\
&\quad + \frac{A |\nabla h|^2}{h - h_{\tau}} - A \sum_j h_j  \psi_j  + A \sum_i b^{ii} h_i^2 \\
&\quad - A \sigma_1 + 2 n A  h - A h^2 \sum_i b^{ii}.
}
We also have
\eq{
\psi_j = \frac{f_j}{f} + p \frac{h_j}{h}, \q
\psi_{jj} = \frac{f f_{jj} - f_j^2}{f^2} + p \frac{h h_{jj} - h_j^2}{h^2}.
}
Therefore, for some positive constants $c_0, c_1, c_2$, independent of $\tau$, depending only on $n$, $p$, $\theta$, $\min_{\mathcal{C}_\theta} f$, $\|f\|_{C^2(\mathcal{C}_\theta)}$ and $\Sigma_0$, we have
\eq{
\frac{\partial_{\tau} Q}{h - h_{\tau}} &\leq c_0 + c_1 A + (c_2 - A) \sigma_1 + (A |\nabla h|^2 - \sigma_1) \sum_i b^{ii}\\
&=c_0 + c_1 A + (c_2 - A) \left(Q-\frac{A}{2}|\nabla h|^2\right) + \left(\frac{3A}{2} |\nabla h|^2 - Q\right) \sum_i b^{ii},
}
where we used \autoref{s3:C0-estimate}, \autoref{C1}, and
\eq{
h - h_{\tau} = f \cK h^{p} \frac{(n+1)V(\widehat{\Sigma_{0}})}{\int_{\mathcal{C}_\theta} f h^{p} \, d\xi} \geq c
}
for some positive constant $c$, independent of $\tau$, due to \autoref{lemma-Gauss-curv-lower}. Now take $A = 2 c_2$. If $Q(\xi_0,\tau)$ is very large, then we have
$
\partial_{\tau} Q< 0.
$
\end{proof}

\subsection{Regularity Estimates II}\label{Sec4-2}
In this subsection, we prove the uniform $C^k$ estimates of solutions to the flow \eqref{GX-p>n+1} for $p > n+1$, without the evenness assumption.

\begin{lemma}\label{s3:C0-estimate-p>n+1}
Let $p > n+1$ and $\theta \in (0,\frac{\pi}{2})$. Suppose $\Sigma_{\tau}$ is a smooth, strictly convex solution to the flow \eqref{GX-p>n+1}. Then there exists a positive constant $C$, depending only on $n$, $p$, $\theta$, $\min_{\mathcal{C}_\theta} f$, $\max_{\mathcal{C}_\theta} f$ and $\Sigma_0$, such that
\eq{\label{s3:C0-estimate-bound-p>n+1}
\frac{1}{C} \leq h(\xi,\tau) \leq C, \quad \forall (\xi,\tau) \in \mathcal{C}_\theta \times (0,\infty).
}
\end{lemma}

\begin{proof}
A direct calculation yields
\eq{\label{s5:evol-u}
\partial_{\tau}\left( \frac{h}{\ell} \right) &= \frac{h}{\ell} \left( -h^{p-1} f \sigma_n^{-1} + 1 \right),
}
and
\eq{\label{s5:bdry-u}
\nabla_\mu \left( \frac{h}{\ell} \right) = 0, \quad \text{on } \partial \mathcal{C}_\theta,
}
where we used that $\nabla_\mu \ell = \cot \theta\, \ell$ on $\partial \mathcal{C}_\theta$.

We claim that for any fixed $\tau > 0$, if the function $\frac{h}{\ell}(\cdot,\tau)$ attains its minimum at a point $\xi_0 \in \mathcal{C}_\theta$, then at this point,
\eq{
\nabla \left( \frac{h}{\ell} \right) = 0, \quad \nabla^2 \left( \frac{h}{\ell} \right) \geq 0.
}
If $\xi_0 \in \mathcal{C}_\theta \setminus \partial \mathcal{C}_\theta$, then the claim follows directly. We assume that $\xi_0 \in \partial \mathcal{C}_\theta$. For any $v \in T_{\xi_0} \partial \mathcal{C}_\theta$, we have
\eq{
0 = \nabla_v  \frac{h}{\ell} .
}
Combining this with \eqref{s5:bdry-u}, we have $\nabla \frac{h}{\ell} = 0$ at $\xi_0$.

Choose a local orthonormal frame $\{e_i\}_{i=1}^n$ around $\xi_0$ such that $e_n = \mu$ and the matrix $b_{ij}|_{(\xi_0,\tau)}$ is diagonal. In this frame, the Hessian $\nabla^2_{ij} \left( \frac{h}{\ell} \right)(\xi_0,\tau)$ is also diagonal. From the Taylor expansion of $\left( \frac{h}{\ell} \right)(\gamma(s),\tau)$, it follows that
\[
\left. \frac{d^2}{ds^2} \right|_{s=0} \left( \frac{h}{\ell} \right)(\gamma(s),\tau) \geq 0 \implies \nabla^2_{\mu\mu} \left( \frac{h}{\ell} \right)(\xi_0,\tau) \geq 0,
\]
where $\gamma$ is the geodesic with $\gamma(0) = \xi_0$ and $\gamma'(0) = -\mu$.

Next, we show that at $\xi_0$, $b_{ij} \geq \frac{h}{\ell} \delta_{ij}$. To this end, note that at the point $\xi_0$, we have
\begin{align}
0 \leq \nabla^2_{ij} \frac{h}{\ell}  &= \frac{\nabla^2_{ij} h}{\ell} + \frac{2h}{\ell^3} \nabla_i \ell \nabla_j \ell - \frac{\nabla^2_{ij} \ell}{\ell^2} h - \frac{\nabla_j h \nabla_i \ell}{\ell^2} - \frac{\nabla_i h \nabla_j \ell}{\ell^2} \\
&= \frac{b_{ij}}{\ell} - \frac{1}{\ell} \left( \nabla_i \frac{h}{\ell}  \nabla_j \ell + \nabla_j \frac{h}{\ell}  \nabla_i \ell \right) - \frac{h}{\ell^2} \delta_{ij} \\
&= \frac{b_{ij}}{\ell} - \frac{h}{\ell^2} \delta_{ij},
\end{align}
where we used $\nabla \left( \frac{h}{\ell} \right) = 0$ at $\xi_0$ and that $\nabla^2_{ij} \ell + \ell \delta_{ij} = \delta_{ij}$.

Since $\sigma_n$ is monotonically increasing in its argument, it follows that
\eq{\label{s5:sigma_n-monotone}
\sigma_n(b_{ij}) \geq \sigma_n \left( \frac{h}{\ell} \delta_{ij} \right) = \left( \frac{h}{\ell} \right)^n.
}
Put $\varphi(\tau) = \min_{\mathcal{C}_\theta} \left( \frac{h}{\ell} \right)(\cdot,\tau)$. Substituting \eqref{s5:sigma_n-monotone} into \eqref{s5:evol-u}, we find
\eq{
\frac{d\varphi}{d\tau}  &\geq \varphi \left( 1 - \varphi^{p-n-1} f \ell^{p-1} \right).
}
Either $\varphi(\tau) \geq \left( \max_{\mathcal{C}_\theta} (f \ell^{p-1}) \right)^{-\frac{1}{p-n-1}}$ or $\frac{d\varphi}{d\tau} > 0$ (whenever it is differentiable). Therefore, $\frac{h}{\ell}$ remains uniformly bounded below away from zero.
Similarly, we can show that $\frac{h}{\ell}$ is uniformly bounded above.
\end{proof}

Based on the above $C^0$ estimate, the $C^1$ and $C^2$ estimates of solutions to the flow \eqref{GX-p>n+1} follow directly from an argument almost identical to that in \autoref{C1}--\autoref{s4:lem-C2*}. We obtain the following result.

\begin{lemma}\label{s4:lem-C2*non-even}
Let $p > n+1$ and $\theta \in (0, \frac{\pi}{2})$. Suppose $\Sigma_{\tau}$ is a smooth, strictly convex solution to the flow \eqref{GX-p>n+1}. Then there exists a positive constant $C$, depending only on $n$, $p$, $\theta$, $\min_{\mathcal{C}_\theta} f$, $\|f\|_{C^2(\mathcal{C}_\theta)}$ and $\Sigma_0$, such that
\eq{
|\nabla h(\cdot,\tau)| \leq C, \quad \frac{1}{C} \leq \kappa_i(\cdot,\tau)\leq C.
}
\end{lemma}

\section{Convergence}
\label{Sec5}
Standard parabolic theory with Neumann boundary conditions guarantees the short-time existence for the flows \eqref{eq:un-normalized_flow} and \eqref{GX-p>n+1}. Moreover, combining \cite[Theorems 6.1, 6.4, 6.5]{D88}, \cite[Theorem 14.23]{Lie96} with our uniform estimates for the normalized solution of the flow \eqref{GX}, we find that the solution to the unnormalized flow \eqref{eq:un-normalized_flow} satisfies $V(\widehat{\cM_t})\to 0$ as $t\to \cT$. Now, from the definition of the time parameter $\tau$ it follows that $\tau(t) \to \infty$ as $t \to \mathcal{T}$.

\begin{proof}[Proof of \autoref{MTO}]
By the uniform estimates established in \autoref{Sec4-1},  the solution $h(\cdot,\tau)$ to \eqref{Gh} enjoys uniform $C^k$-norm bound for each $k \in \bbN$. Hence, by the Arzel\`{a}--Ascoli theorem and \autoref{monJ2}, we can extract a subsequence of times, $\{ \tau_i \}_{i \geq 1} \subset (0,\infty)$, such that $h(\cdot,\tau_i)$ converges smoothly to an even strictly convex function $h$ such that $h$ satisfies \eqref{s1:static-eq}, where $\Sigma$ is the even, smooth, and strictly convex capillary hypersurface  whose capillary support function is $h$.
\end{proof}

\begin{proof}[Proof of \autoref{MTO-noneven}]
Due to \autoref{monJ-noneven} and the uniform estimates shown in \autoref{Sec4-2}, a  subsequence of the solution to \eqref{h-flow 2} converges to a smooth, strictly convex solution of \eqref{Lp-MP-smooth}.
\end{proof}

\section*{Acknowledgment}
J. Hu and Ivaki were supported by the Austrian Science Fund (FWF) under Project P36545.
Y. Hu was supported by the National Key Research and Development Program of China 2021YFA1001800.

\vspace{10mm}
\textsc{Institut f\"{u}r Diskrete Mathematik und Geometrie,\\ Technische Universit\"{a}t Wien,\\ Wiedner Hauptstra{\ss}e 8-10, 1040 Wien, Austria,\\} 
\email{\href{jinrong.hu@tuwien.ac.at}{jinrong.hu@tuwien.ac.at}}

	\vspace{5mm}
	\textsc{School of Mathematical Sciences, Beihang University,\\ Beijing 100191, China,\\}
	\email{\href{mailto:huyingxiang@buaa.edu.cn}{huyingxiang@buaa.edu.cn}}
	
		\vspace{5mm}
\textsc{Institut f\"{u}r Diskrete Mathematik und Geometrie,\\ Technische Universit\"{a}t Wien,\\ Wiedner Hauptstra{\ss}e 8-10, 1040 Wien, Austria,\\} \email{\href{mailto:mohammad.ivaki@tuwien.ac.at}{mohammad.ivaki@tuwien.ac.at}}

\end{document}